\documentclass[12pt]{article}

\usepackage{amssymb}
\usepackage{amsthm}
\usepackage{amsmath}
\usepackage{amsfonts}

\newtheorem{thm}{Theorem}[section]
\newtheorem{cor}[thm]{Corollary}
\newtheorem{lem}[thm]{Lemma}
\newtheorem{trm}[thm]{Theorem}
\newtheorem{rem}[thm]{Remark}

\begin{document}

\title{On the gap of finite metric spaces of $p$-negative type}
\author{Reinhard Wolf}
\date{}

\maketitle
\begin{center} Universit\"{a}t Salzburg, Fachbereich Mathematik, Austria
\end{center}

\begin{abstract}
Let $(X,d)$ be a metric space of $p$-negative type. Recently I.
Doust and A. Weston introduced a quantification of the $p$-negative
type property, the so called gap $\Gamma$ of $X$. \newline This
paper gives some formulas for the gap $\Gamma$ of a finite metric
space of strict $p$-negative type and applies them to evaluate
$\Gamma$ for some concrete finite metric spaces.
\end{abstract}

\section{Introduction}
\label{}

Let $(X,d)$ be a metric space and $p\geq0$. Recall that $(X,d)$ has
$p$-negative type if for all natural numbers $n$, all $x_1, x_2,
\ldots, x_n$ in $X$ and all real numbers $\alpha_1, \alpha_2,
\ldots, \alpha_n$ with $\alpha_1+\alpha_2+\ldots+\alpha_n=0$ the
inequality
    \begin{displaymath}
    \sum^{n}_{i,j=1}\alpha_i\alpha_jd(x_i,x_j)^p\leq0
    \end{displaymath}
holds.\\

Moreover if $(X,d)$ has $p$-negative type and
    \begin{eqnarray*}
    \sum^{n}_{i,j=1}\alpha_i\alpha_jd(x_i,x_j)^p=0, \\
    \textnormal{together with} \ x_i\neq x_j, \ \textnormal{for all} \ i\neq j
    \end{eqnarray*}
    \\
implies $\alpha_1=\alpha_2=\ldots=\alpha_n=0$, then $(X,d)$ has
strict $p$-negative type. ($d(x,y)^0$ is defined to be $0$ if
$x=y$).\\

Following [2] and [3] we define the $p$-negative type gap
$\Gamma^p_X$ (=$\Gamma$ for short) of a $p$-negative type metric
space $(X,d)$ as the largest nonnegative constant, such that
    \begin{displaymath}
    \frac{\Gamma}{2}\left(\sum^{n}_{i=1}|\alpha_i|\right)^2+\sum^n_{i,j=1}\alpha_i\alpha_jd(x_i,x_j)^p\leq0
    \end{displaymath}
holds for all natural numbers $n$, all $x_1,x_2,\ldots,x_n$ in $X$
and all real numbers $\alpha_1,\alpha_2,\ldots,\alpha_n$ with
$\alpha_1+\alpha_2+\ldots+\alpha_n=0$.\\

For basic information on $p$-negative type spaces (1-negative type
spaces are also known as quasihypermetric spaces) see for example
[2], [4], \ldots, [9].\\

This paper explores formulas for the $p$-negative type gap of a
finite $p$-negative type metric space and applies them to calculate
the 1-negative type gap $\Gamma$ of a cycle graph with $n$ vertices
(considered as a finite metric space with the usual path length
metric). It is shown that
    \begin{displaymath}
    \Gamma=\left\{ \begin{array}{cll}
    0 & & , \textnormal{if $n$ is even}\\
    \frac{1}{2} & \cdot\frac{n}{n^2-2n-1} & , \textnormal{if $n$ is
    odd.}
    \end{array} \right.
    \end{displaymath}
Moreover we present short proofs for the evaluation of the
1-negative type gap of a finite discrete metric space, done by A.
Weston in [9], and of a finite metric tree, done by I. Doust and A.
Weston in [2]. I. Doust and A. Weston showed the surprising result,
that the gap of a finite metric tree only depends on the weights
associated to the edges of the tree.

\section{Notation}
\label{}

For a given real $m\times n$ matrix $A$ we denote by $A^T$ the
transposed matrix of $A$ and by $A^{-1}$ the inverse matrix of $A$,
if it exists. Elements $x$ in $\mathbb{R}^n$ are interpretated as
column vectors, so $x^T=(x_1,x_2, \ldots,x_n)$. The canonical inner
product of two elements $x,y$ in $\mathbb{R}^n$ is given by $(x|y)$
and the canonical unit vectors are denoted by $e_1,e_2,\ldots,e_n$.
The element $\underline{1}$ in $\mathbb{R}^n$ is defined as
$\underline{1}^T=(1,1,\ldots, 1)$. The linear span and convex hull
of a subset $M$ in $\mathbb{R}^n$ are denoted by $[M]$ and conv $M$.
Further let $\ker T$ be the kernel of a given linear map $T$.\\

If $E$ is a linear subspace of $\mathbb{R}^n$ and $\|.\|$ is a norm
on $E$ we denote by $\|.\|^*$ the dual norm of $\|.\|$ on $E$ with
respect to the canonical inner product, e.g.
    \begin{displaymath}
    \|x\|^*=\sup_{y\in E,\|y\|\leq1}|(x|y)|
    \end{displaymath}
For $p\geq1$ we let $\|x\|_p$ be the usual $p$-norm of some element
$x$ in $\mathbb{R}^n$. For a given real symmetric \newline $n\
\times\  n$  matrix $A$ which is positive semi-definite on a linear
subspace $E$ of $\mathbb{R}^n ((Ax|x)\geq0, \ \textnormal{for all} \
x \ \textnormal{in}\  E)$ we define the resulting semi-inner product
on $E$ by
    \begin{displaymath}
    (x|y)_A=(Ax|y); \quad x,y \ \textnormal{in} \ E.
    \end{displaymath}
Further the semi-norm $\|x\|_A$ of some element $x$ in $E$ is given
by
    \begin{displaymath}
    \|x\|^2_A=(Ax|x).
    \end{displaymath}

\section{The results}
\label{}

Let $E$ be a linear subspace of $\mathbb{R}^n$ and $A$ be a real
symmetric $n\times n$ matrix of negative (strict negative) type on
$E$, e.g.
    \begin{eqnarray*}
    (Ax|x)\leq0,\quad \textnormal{for all $x$ in $E$ and}\\
    (Ax|x)<0,\quad \textnormal{for all $x\neq 0$ in $E$ resp.}
    \end{eqnarray*}
For further discussion it is useful to define the negative type gap
$\Gamma_{A,E}$ (=$\Gamma$ for short) of $A$ on $E$ as the largest
nonnegative constant, such that
    \begin{displaymath}
    \frac{\Gamma}{2}\|x\|^2_1+(Ax|x)\leq0
    \end{displaymath} holds for all $x$ in $E$. If $A$ is of strict negative type on
$E$, this is equivalent to
    \begin{displaymath}
   (\ast) \quad \left(\frac{2}{\Gamma}\right)^{\frac{1}{2}}=
    \sup_{x\in E,\|x\|_{-A}\leq1} \|x\|_1.
    \end{displaymath}
Note that in this case $\Gamma>0$, since norms are equivalent on
$\mathbb{R}^n$.\\

To continue fix some $u\neq0$ in $\mathbb{R}^n$ and let
    \begin{displaymath}
    F_\alpha=\{x\in\mathbb{R}^n|(x|u)=\alpha\}, \quad \alpha \
    \textnormal{in} \
    \mathbb{R}.
    \end{displaymath}
For short let $F=F_0$. \\
Furthermore let $A$ be a real symmetric $n\times n$ matrix of
negative type on $F$ and not of negative type on $\mathbb{R}^n$
(note that this condition is equivalent to $A$ is of negative type
on $F$ and there is some $w$ in $F_1$ with $(Aw|w)>0$). Following
some ideas of [5], \ldots, [8] we define $M_{F_1}(A)(=M \
\textnormal{for short})$ as
    \begin{displaymath}
    M=\sup_{x\in F_1} \ (Ax|x)(>0).
    \end{displaymath}

\begin{thm} Let $A$ be a real symmetric $n\times n$ matrix of negative
type on $F$ and not of negative type on $\mathbb{R}^n$. Further let
$M=\sup_{x\in F_1} \ (Ax|x)$.\\
We have
\begin{enumerate}
\item $A$ is of strict negative type on $F$ if and only if $A$ is
nonsingular and $(A^{-1}u|u)\neq0$.
\item If $A$ is of strict negative type on $F$ then we have
\begin{enumerate}
\item there is a unique (maximal) element $z$ in $F_1$ such that
$M=(Az|z)$.
\item $Az=Mu$ and $M=(A^{-1}u|u)^{-1}$.
\end{enumerate}
\end{enumerate}
\end{thm}

\begin{proof}
Assume first that $A$ is of strict negative type on $F$ and let
$Ax=0$ for some $x$ in $\mathbb{R}^n$. Choose some $w$ in $F_1$ with
$(Aw|w)>0$. If $(x|u)\neq0$, we get
$\displaystyle(A(w-\frac{x}{(x|u)})|w-\frac{x}{(x|u)})\leq0$ and
hence $(Aw|w)\leq0$, a contradiction. Therefore we have $x$ in $F$
and so $x=0$, which shows that $A$ is nonsingular. Now let $y$ in
$\mathbb{R}^n$ be the unique element with $Ay=u$.\\
If $y$ in $F$ we obtain $(Ay|y)=0$ and hence $y=0$, a contradiction.
Therefore we have $(A^{-1}u|u)=(y|u)\neq0$. Let $\displaystyle
z=\frac{1}{(y|u)}y, z$ in $F_1$. So $\displaystyle
Az=\frac{1}{(y|u)}u$ and it follows that for all $x$ in $F_1$,
$x\neq z$ we get $(A(x-z)|x-z)<0$ and hence $\displaystyle
(Ax|x)<\frac{1}{(y|u)}$. Since $\displaystyle(Az|z)=\frac{1}{(y|u)}$
we get $\displaystyle M=\frac{1}{(y|u)}=\frac{1}{(A^{-1}u|u)}$ and
$Az=Mu$.\newline It remains to show that $A$ nonsingular and
$(A^{-1}u|u)\neq0$ implies that $A$ is of strict negative type on
$F$. \newline Let $(Ax|x)=0$ for some $x$ in $F$. Since
$|(Ax|y)|^2\leq(Ax|x)(Ay|y)$ for all $y$ in $F$ we get
    \begin{displaymath}
    Ax=\lambda u, \quad \textnormal{for some}\  \lambda\  \textnormal{in}\
     \mathbb{R}.
    \end{displaymath}
Hence $0=(x|u)=\lambda(A^{-1}u|u)$ \\
and therefore $\lambda=0$, which implies $Ax=0$ and so $x=0$.
\end{proof}

The following application was done in [8] (Theorem 2.11) for finite
metric spaces of 1-negative type (finite quasihypermetric spaces).

\begin{cor}
Let $(X,d)$ with $X=\{x_1,x_2,\ldots,
    x_n\}$ be a finite metric space of $p$-negative type of at
least two points. $(X,d)$ is of strict $p$-negative type if and only
if
    \begin{displaymath}
    A=\left(d(x_i, x_j)^p\right)^n_{i,j=1}
    \end{displaymath}
is nonsingular and $(A^{-1}\underline{1}|\underline{1})\neq0$.
\end{cor}

\begin{proof}
Take $u=\underline{1}$, and note that
$\displaystyle(Aw|w)=\frac{d(x_i,x_j)^p}{2}>0$, for $\displaystyle
w=\frac{e_i+e_j}{2}$ in $F_1$ and $x_i\neq x_j$ and so we are done
by Theorem 3.1, part 1.
\end{proof}

Now let $A$ be of strict negative type on
    \begin{displaymath}
    F=\{x\in \mathbb{R}^n|(x|u)=0\}, \quad u\neq0 \ \textnormal{in}
    \ \mathbb{R}^n.
    \end{displaymath}
By Theorem 3.1 we know that $M=\sup_{x\in F_1}(Ax|x)$ is finite and
there is a unique (maximal) element $z$ in $F_1$, such that
$M=(Az|z)$ and $Az=Mu$. \\
Define
    \begin{displaymath}
    C=Muu^T-A.
    \end{displaymath}
Again by Theorem 3.1, $C$ is positive semi-definite on
$\mathbb{R}^n$ with $\ker C=[z]$. Therefore we can extend the inner
product $(.|.)_{-A}$ defined on $F$ to a semi-inner product on
$\mathbb{R}^n$ given by
    \begin{displaymath}
    (x|y)_C=(Cx|y), \quad \textnormal{for} \ x,y \ \textnormal{in} \
    \mathbb{R}^n.
    \end{displaymath}
Furthermore we define
    \begin{displaymath}
    B=\frac{1}{M}zz^T-A^{-1}.
    \end{displaymath}
Since $(BAx|Ax)=(Cx|x)$, for all $x$ in $\mathbb{R}^n$, it follows
that
$B$ is positive semi-definite on $\mathbb{R}^n$ with $\ker B=[u]$.\\

Before formulating the next lemma, dealing with dual norms on $F$,
we define for $x^T=(x_1, x_2, \ldots, x_n)$ in $\mathbb{R}^n$ the
oscillation of $x$ with respect to $u$ as
    \begin{displaymath}
    o(x)=\max\left(\max_{i,j\in \ \textnormal{supp}\
    u}\frac{|u_ix_j-u_jx_i|}{|u_i|+|u_j|},\max_{i \notin\
    \textnormal{supp}
    \ u}|x_i|\right),
    \end{displaymath}
where
    \begin{displaymath}
    \textnormal{supp} \ u=\{1\leq i\leq n|u_i\neq0\}.
    \end{displaymath}
Note that $o(.)$ defines a semi-norm on $\mathbb{R}^n$ and $o(x)=0$
if and only if $x$ in $[u]$.

\begin{lem}
We have
\begin{enumerate}
\item The dual norm of $\|.\|_1$ on $F$ is given by $\|.\|^*_1=o(x)$, for all $x$ in $F$.
\item The dual norm of $\|.\|_{-A}$ on $F$ is given by
$\|x\|^*_{-A}=\|x\|_B$, for all $x$ in $F$.
\item $\{x\in F| \|x\|^*_1\leq1\}=\ \textnormal{conv} \ E$, where
    \begin{displaymath}
    E=\left\{x-\frac{(x|u)}{\|u\|^2_2}u, x\in\{-1,1\}^n\right\}.
    \end{displaymath}
\end{enumerate}
\end{lem}

\begin{proof} We have
\begin{enumerate}
\item[ad. 1] Intersecting $F$ with the edges of the cross-polytope
conv $\{\pm e_i, 1\leq i\leq n\}$ leads to the set $M$ of extreme
points of $\{x\in F| \|x\|_1\leq1\}$, where
    \begin{displaymath}
    M=\left\{\pm\frac{u_ie_j-u_je_i}{|u_i|+|u_j|}, i,j\in \ \textnormal{supp}\ u; \pm e_i, i\notin \ \textnormal{supp} \ u\right\}
    \end{displaymath}
and hence
    \begin{displaymath}
    \|x\|^*_1=o(x), \quad \textnormal{for all}\  x \ \textnormal{in} \ F.
    \end{displaymath}
\item[ad. 2] Let $x,y$ be in $F$.
    \begin{eqnarray*}
    (x|y)^2  =  (A^{-1}x|Ay)^2 & = & (A^{-1}x|y)^2_{-A} =\\
     =  (A^{-1}x|y)^2_C & \leq & (CA^{-1}x|A^{-1}x)(Cy|y) =\\
     =  (Bx|x)(y|y)_{-A} & = & (Bx|x) \|y\|^2_{-A}.
    \end{eqnarray*}
Therefore we get $\displaystyle \|x\|^*_{-A}
\leq(Bx|x)^{\frac{1}{2}}$. Let $\displaystyle
y_0=\frac{(x|z)}{M}z-A^{-1}x$. Note that $y_0$ in $F$ and
$\|y_0\|^2_{-A}=(Bx|x)$ and so
    \begin{displaymath}
    \|x\|^*_{-A}\geq\left(x|\frac{y_0}{(Bx|x)^{\frac{1}{2}}}\right)=(Bx|x)^{\frac{1}{2}}.
    \end{displaymath}
\item[ad. 3] Let $x$ be in $\{-1,1\}^n$.
Now $\displaystyle
o\left(x-\frac{(x|u)}{\|u\|^2_2}u\right)=o(x)\leq1$ and hence
$E\subseteq\{x\in F| \|x\|^*_1\leq1\}$. For a given $y$ in $F$
choose $\alpha^T=(\alpha_1, \alpha_2, \ldots, \alpha_n)$ in
$\{-1,1\}^n$ such that $\|y\|_1=(y|\alpha)$. So
$\displaystyle\left(y|\alpha-\frac{(\alpha|u)u}{\|u\|^2_2}\right)=(y|\alpha)=\|y\|_1$
and hence $\|y\|_1=\sup_{x\in E}(y|x)$, e.g. $E=-E$ is a norming
set. Therefore we obtain conv $E=\{x\in F | \|x\|^*_1\leq1\}$.
\end{enumerate}
\end{proof}

\begin{trm}
Let $u\neq0$ be in $\mathbb{R}^n$ and
$F=\{x\in\mathbb{R}^n|(x|u)=0\}$. Further let $A$ be a real
symmetric $n\times n$ matrix of strict negative type on $F$, and not
of
negative type on $\mathbb{R}^n$.\\
The gap $\Gamma_A(F)(=\Gamma)$ of $A$ on $F$ is given by
$\displaystyle\Gamma=\frac{2}{\beta}$, where
\begin{enumerate}
\item    \begin{eqnarray*}
    \beta=
     \sup_{x\in F,o(Ax)\leq1}(-Ax|x)
     \end{eqnarray*}
\item \begin{eqnarray*}
    \beta =
    \max_{x\in\{-1,1\}^n}(Bx|x)
    \end{eqnarray*}
where
    \begin{displaymath}
    B=(A^{-1}u|u)^{-1}(A^{-1}u)(A^{-1}u)^T-A^{-1}.
    \end{displaymath}
\item $\beta= \|B\|$, where $B$ is defined as in 2. and viewed as a
linear operator from $(\mathbb{R}^n, \|.\|_\infty)$ to
$(\mathbb{R}^n, \|.\|_1)$.
\end{enumerate}
\end{trm}

\begin{proof} We have
\begin{enumerate}
\item[ad. 1] By formula $(*)$ at the beginning of the chapter we have
$\displaystyle\Gamma=\frac{2}{\beta}$ with
    \begin{eqnarray*}
    \beta^{\frac{1}{2}}=
    \sup_{x\in F,\|x\|_{-A}\leq1}\|x\|_1.
    \end{eqnarray*}
By Lemma 3.3, part 1,2 we get
    \begin{displaymath}
    \beta^{\frac{1}{2}}= \sup_{x\in F,\|x\|_{-A}\leq1}\|x\|_1 =  \sup_{x\in F,\|x\|^*_1\leq1} \|x\|^*_{-A} =  \sup_{x\in F,o(x)\leq1} \|x\|_B
    \end{displaymath}
Recall that $\ker B=[u]$ and $o(x+\lambda u)=o(x)$, for all $x$ in
$\mathbb{R}^n$ and $\lambda$ in $\mathbb{R}$. Hence
    \begin{displaymath}
    \beta= \sup_{x\in F,o(x)\leq1}\|x\|^2_B = \sup_{x\in\mathbb{R}^n,o(x)\leq1} \|x\|^2_B = \sup_{y\in\mathbb{R}^n,o(Ay)\leq1} \|Ay\|^2_B,
    \end{displaymath}
since $A$ is nonsingular by Theorem 3.1, part 1. Since
$\|Ay\|^2_B=(Cy|y)$ where $C=Muu^T-A, Az=Mu$ (see Theorem 3.1, part
2) we get
    \begin{displaymath}
    \beta= \sup_{y\in\mathbb{R}^n,o(Ay)\leq1} (Cy|y).
    \end{displaymath}
Since $z$ is not in $F$, we can write each $y$ in $\mathbb{R}^n$ as
$y=f+\lambda z$, for some $f$ in $F$ and $\lambda$ in $\mathbb{R}$.
Recall that $\ker C=[z]$ and $o(Ay)=o(Af+\lambda Mu)=o(Af)$ and so
    \begin{displaymath}
    \beta= \sup_{y\in\mathbb{R}^n,o(Ay)\leq1}(Cy|y)= \sup_{x\in
    F,o(Ax)\leq1}(-Ax|x).
    \end{displaymath}
\item[ad. 2] From above we have
    \begin{displaymath}
    \beta= \sup_{x\in F,\|x\|^*_1\leq1}\|x\|^2_B = \max_{x\in E}(Bx|x)
    \end{displaymath}
by Lemma 3.3, part3, where
    \begin{displaymath}
    E=\left\{x-\frac{(x|u)}{\|u\|^2_2}u, x\in\{-1,1\}^n\right\}.
    \end{displaymath}
Again using the fact, that $\ker B=[u]$ we get
    \begin{displaymath}
    \beta=\max_{x\in\{-1,1\}^n} (Bx|x).
    \end{displaymath}
\item[ad. 3] Recall that $B$ is positive semi-definite on
$\mathbb{R}^n$ and hence for all $x,y$ in $\{-1,1\}^n$ we get
$(Bx|y)^2\leq(Bx|x)(By|y)$ and so
    \begin{displaymath}
    \beta=\max_{x,y\in\{-1,1\}^n}(Bx|y)=\max_{x\in\{-1,1\}^n}\|Bx\|_1=\|B\|.
    \end{displaymath}
\end{enumerate}
\end{proof}

Now let $(X,d)$ be a finite metric space of strict $p$-negative
type, $X=\{x_1, x_2, \ldots,x_n\}, n\geq2$. Let
$A=\left(d(x_i,x_j)^p\right)^n_{i,j=1}$ and $u=\underline{1}$. By
corollary 3.2 we know that $A$ is nonsingular and
$(A^{-1}\underline{1}|\underline{1})\neq0$. Recall that
$\displaystyle(Aw|w)=\frac{d(x_i,x_j)^p}{2}>0$, for $\displaystyle
w=\frac{e_i+e_j}{2}, i\neq j$. Further observe that
$u=\underline{1}$ implies $\displaystyle
o(x)=\max_{i,j}\frac{|x_i-x_j|}{2}$ and $x$ is in $\{-1,1\}^n$ if
and only $x+\underline{1}$ is in $\{0,2\}^n$ with
$o(x)=o(x+\underline{1})$. Applying Theorem 3.3 we get

\begin{trm}
Let $(X,d)$ with $X=\{x_1,x_2, \ldots, x_n\}$ be a finite metric
space of strict $p$-negative type of at least two points. Let
    \begin{displaymath}
A=\left(d(x_i,x_j)^p\right)^n_{i,j=1}
    \end{displaymath}
The $p$-negative type gap $\Gamma$ of $X$ is given by $\displaystyle
\Gamma=\frac{2}{\beta}$, where
\begin{enumerate}
\item
    \begin{displaymath}
    \beta=\sup\left\{(-Ax|x)|x_1+x_2+\ldots+x_n=0 \ \textnormal{and} \ |(Ax|e_i-e_j)|\leq2, \ \textnormal{for all} \  1\leq i,j\leq n\right\}
    \end{displaymath}
\item
    \begin{displaymath}
    \beta=\max_{x\in\{-1,1\}^n}(Bx|x)=4\max_{x\in\{0,1\}^n}(Bx|x), \
    \textnormal{where} \ B=(A^{-1}\underline{1}|\underline{1})^{-1}(A^{-1}\underline{1})(A^{-1}\underline{1})^T-A^{-1}
    \end{displaymath}
\item
    \begin{displaymath}
    \beta=\|B\|,
    \end{displaymath}
where $B$ is defined as in 2. and viewed as a linear operator from
$(\mathbb{R}^n,\|.\|_\infty)$ to $(\mathbb{R}^n, \|.\|_1)$.
\end{enumerate}
\end{trm}

\begin{cor}
Let $n$ be a natural number greater or equal to $3$ and let $C_n$ be
the cycle graph with $n$ vertices, viewed as a finite metric space,
equipped with the usual path metric. Then we have
\begin{enumerate}
\item $C_n$ is of 1-negative type and of strict 1-negative type if
and only if $n$ is odd.
\item The 1-negative type gap $\Gamma$ of $C_n$ is given by
    \begin{displaymath}
    \Gamma=\left\{\begin{array}{ll}
    0, & n \ \textnormal{even}\\
    \frac{1}{2}\frac{n}{n^2-2n-1}, & n \ \textnormal{odd}.
    \end{array}\right.
    \end{displaymath}
\end{enumerate}
\end{cor}

\begin{proof}
Take the vertices $\{x_1, x_2, \ldots,x_n\}$ of a regular $n$-gon on
a circle $C$ of radius $\displaystyle r=\frac{n}{2\pi}$. It is
evident, that $C_n$ can be viewed as the subspace
$\{x_1,x_2,\ldots,x_n\}$ of the metric space $(C,d)$, where $d$ is
the arc-length metric on $C$. It is shown in [4] (see Theorem 4.3
and Theorem 9.1) that $(C,d)$ is of 1-negative type and a finite
subspace of $(C,d)$ is of strict 1-negative type if and only if this
subspace contains at most one pair of antipodal points. Hence part 1
follows at once. The definition of $\Gamma$ implies that $\Gamma=0$
if $n$ is even, so
let us assume that $n=2k+1$, for some $k$ in $\mathbb{N}$.\\

Let $A$ be the distance matrix of $C_n=C_{2k+1}$. It is shown in [1]
(Theorem 3.1) that $A^{-1}$ is given by
    \begin{displaymath}
    A^{-1}=-2I-C^k-C^{k+1}+\frac{2k+1}{k(k+1)} \underline{1} \ \underline{1}^T,
    \end{displaymath}
where $I$ is the identity matrix and $C$ is the matrix (with respect
to the canonical bases) of the linear map on $\mathbb{R}^n$, which
sends each $x^T=(x_1,x_2,\ldots,x_n)$ to $(x_2,x_3,\ldots,x_n,x_1)$.
Now $\displaystyle
A^{-1}\underline{1}=\frac{1}{k(k+1)}\underline{1}$ and so $B$ (as
defined in Theorem 3.4) is given by $\displaystyle
B=2I+C^k+C^{k+1}-\frac{4}{2k+1}\underline{1} \ \underline{1}^T$ and
so
    \begin{displaymath}
    (Bx|x)=2\|x\|^2_2-\frac{4}{2k+1}(x|\underline{1})^2+2(x_1x_{k+1}+\ldots+x_{k+1}x_{2k+1}+x_{k+2}x_1+\ldots+x_{2k+1}x_k),
    \end{displaymath}
for each $x^T=(x_1,x_2,\ldots,x_{2k+1})$ in $\mathbb{R}^{2k+1}$.\\

Now let $x$ be in $\{0,1\}^{2k+1}$ and let $s=|\{1\leq
i\leq2k+1|x_i=1\}|$. In the case $s=0$ and $s=2k+1$ we get
$(Bx|x)=0$, so assume that $1\leq s\leq2k$. Since
    \begin{displaymath}
    x_1x_{k+1}+\ldots+x_{k+1}x_{2k+1}+x_{k+2}x_1+\ldots+x_{2k+1}x_k<\|x\|^2_2=s,
    \end{displaymath}
we get
    \begin{displaymath}
    (Bx|x)\leq2s-\frac{4}{2k+1}s^2+2(s-1)=2\left(2s-1-\frac{2}{2k+1}s^2\right).
    \end{displaymath}
It follows immediately, that
    \begin{displaymath}
    \max_{1\leq s\leq2k}\left(2s-1-\frac{2}{2k+1}s^2\right)=\frac{2k^2-1}{2k+1}
    \end{displaymath}
and hence
    \begin{displaymath}
    \max_{x\in\{0,1\}^{2k+1}}(Bx|x)\leq\frac{4k^2-2}{2k+1}.
    \end{displaymath}
On the other hand define $\overline{x}$ in $\{0,1\}^{2k+1}$ as
$\overline{x}^T=(\alpha_1,\alpha_2,\ldots,\alpha_{2k+1})$, with
$\alpha_i=1$ if and only if
$i\in\{1,2,\ldots,m,2m+1,2m+2,\ldots3m+1\}$ if $k=2m$ and
$\alpha_i=1$ if and only if
$i\in\{1,2,\ldots,m,2m+2,2m+3,\ldots,3m+2\}$ if $k=2m+1$, $m$ in
$\mathbb{N}$. In each case $(k=2m, 2m+1)$ we get $\displaystyle
(B\overline{x}|\overline{x})=\frac{4k^2-2}{2k+1}$. Summing up we
have $\displaystyle
\max_{x\in\{0,1\}^{2k+1}}(Bx|x)=\frac{4k^2-2}{2k+1}$ and hence
Theorem 3.4, part 2 implies
$\displaystyle\Gamma=\frac{2k+1}{8k^2-4}=\frac{1}{2}\frac{n}{n^2-2n-1}$.
\end{proof}

\begin{cor}
(=Theorem 3.2 of [9]) Let $(X,d)$ be a finite discrete space
consisting of $n$ points, $n\geq2$. The 1-negative type gap $\Gamma$
of $X$ is given by
    \begin{displaymath}
    \Gamma=\frac{1}{2}\left(\frac{1}{\lfloor\frac{n}{2}\rfloor}+\frac{1}{\lceil\frac{n}{2}\rceil}\right).
    \end{displaymath}
\end{cor}

\begin{proof}
Let $A$ be the distance matrix of $X$. We have $A=\underline{1} \
\underline{1}^T-I$ ($I$ the identity matrix) and hence
$\displaystyle A^{-1}=\frac{1}{n-1}\underline{1} \
\underline{1}^T-I$. So the matrix $B$ defined as in Theorem 3.4 is
given by $\displaystyle B=I-\frac{1}{n}\underline{1} \
\underline{1}^T$. Applying Theorem 4, part 2 we get
    \begin{displaymath}
    \beta=\max_{x\in\{-1,1\}^n}(Bx|x)=n-\frac{1}{n}\min_{x\in\{-1,1\}^n}(x|\underline{1})^2=\left\{\begin{array}{ll}
    n, & n \ \textnormal{even}\\
    n-\frac{1}{n}, & n \ \textnormal{odd}
    \end{array}\right.
    \end{displaymath}
and so $\displaystyle
\Gamma=\frac{2}{\beta}=\frac{1}{2}\left(\frac{1}{\lfloor\frac{n}{2}\rfloor}+\frac{1}{\lceil\frac{n}{2}\rceil}\right)$.
\end{proof}

Recall that for a given finite connected simple graph $G=(V,E)$ and
a given collection $\{w(e),e\in E\}$ of positive weights associated
to the edges of $G$, the graph $G$ becomes a finite metric space,
where the metric is given by the natural weighted path metric on
$G$. A finite metric tree $T=(V,E)$ is a finite connected simple
graph that has no cycles, endowed with the above given edge weighted
path metric. It is shown in [4] (Corollary 7.2) that metric trees
are of strict 1-negative type.\\

\begin{cor}
(=Corollary 4.14 of [2]). Let $T=(V,E)$ be a finite metric tree. The
1-negative type gap $\Gamma$ of $G$ is given by $\displaystyle
\Gamma=\left(\sum_{e\in E}\frac{1}{w(e)}\right)^{-1}$, where $w(e)$
denotes the weight of the edge $e$.
\end{cor}

\begin{proof}
Let $V=\{1,2,\ldots,n\}$ be the set of vertices of $T$. Recall
$A=(d(i,j))^n_{i,j=1}$ be the distance matrix of $T$ and
$(x|y)_{-A}=(-Ax|y)$ be the inner product determined by $A$ on
$F=\{x\in\mathbb{R}^n|x_1+x_2+\ldots+x_n=0\}$. Now take a
2-colouring of the vertices of $T$ by colours $a,b$ and define for
each $e$ in $E$ $z_e=e_i-e_j$, where $i$ and $j$ are adjacent
vertices in $V$ connected by $e$ and $i$ has colour $a$ and $j$ has
colour $b$.\ Note that $z_e$ in $F$ and $\|z_e\|^2_{-A}=2w(e)$.
Since each vertices $r,s$ in $V$ are connected by exactly one path
in $T$, it follows immediately, that
    \begin{displaymath}
    (z_e|e_r-e_s)_{-A}=(d(j,r)-d(i,r))-(d(j,s)-d(i,s))=0,
    \end{displaymath}
if $e$ is not lying on the path connecting $r$ and $s$ and
$(z_e|e_r-e_s)_{-A}\in\{2w(e),-2w(e)\}$ if $e$ is lying on the path
connecting $r$ and $s$, where the value $2w(e)$ is obtained if and
only if $d(j,r)>d(i,r)$. From these observations and $|E|=|V|-1$ it
follows, that $\displaystyle \{x_e=(2w(e))^{-\frac{1}{2}}z_e,e\in
E\}$ is an
orthonormal basis of $F$ with respect to $(.|.)_{-A}$.\\
Now let $x$ be in $F$.
    \begin{displaymath}
    (-Ax|x)=\|x\|^2_{-A}=\sum_{e\in E}|(x|x_e)_{-A}|^2=\sum_{e\in
    E}\frac{1}{2w(e)}|(Ax|e_i-e_j)|^2.
    \end{displaymath}
For $\beta=\sup\left\{(-Ax|x)|x\in F,|(Ax|e_r-e_s)|\leq2 \
\textnormal{for all}\ 1\leq r,s\leq n\right\}$ we get
    \begin{displaymath}
    \beta\leq\sum_{e\in E}\frac{1}{2w(e)}\cdot4=2\sum_{e\in
    E}\frac{1}{w(e)}.
    \end{displaymath}
On the other hand let $\displaystyle \overline{x}=\sum_{e\in
E}\frac{z_e}{w(e)}$ in $F$ and take some $r\neq s$ in $V$. From
above we know, that
    \begin{displaymath}
    (\overline{x}|e_r-e_s)_{-A}=\sum_{e\in
    E}\frac{1}{w(e)}(z_e|e_r-e_s)_{-A}=\sum_{e\in
    P}\frac{1}{w(e)}(z_e|e_r-e_s)_{-A},
    \end{displaymath}
where $P$ is the set of edges lying on the path connecting $r$ and
$s$. By definition of $z_e$ and since colours $a$ and $b$ are
changing each step running from $r$ to $s$, we get
    \begin{displaymath}
    \sum_{e\in P}\frac{1}{w(e)}(z_e|e_r-e_s)_{-A}=2-2+2-+\ldots
    \in\{0,2\}
    \end{displaymath}
or
    \begin{displaymath}
    \sum_{e\in P}\frac{1}{w(e)}(z_e|e_r-e_s)_{-A}=-2+2-2+-\ldots \in
    \{0,-2\}.
    \end{displaymath}
Hence $|(A\overline{x}|e_r-e_s)|=|(\overline{x}|e_r-e_s)_{-A}|\in
\{0,2\}$, for each $r$ and $s$ in $V$ and so
    \begin{displaymath}
    \beta\geq(-A\overline{x}|\overline{x})=\|\overline{x}\|^2_{-A}=\sum_{e\in
    E}|(\overline{x}|x_e)_{-A}|^2=\sum_{e\in
    E}\frac{|(z_e|x_e)_{-A}|^2}{w(e)^2}=2 \sum_{e\in E}\frac{1}{w(e)}.
    \end{displaymath}
Applying Theorem 3.4, part 1 we get
    \begin{displaymath}
    \Gamma=\frac{2}{\beta}=\left(\sum_{e\in
    E}\frac{1}{w(e)}\right)^{-1}.
    \end{displaymath}
\end{proof}

We conclude this paper with the following

\begin{rem}
It is shown in [1] (Theorem 2.1) that the inverse matrix $A^{-1}$ of
the distance matrix $A$ of a finite metric tree is given by
    \begin{displaymath}
    A^{-1}=-\frac{1}{2}L+\left(2\sum_{e\in
    E}w(e)\right)^{-1}\delta\delta^T,
    \end{displaymath}
where $L$ denotes the Laplacian matrix for the weighting of $T$ that
arises by replacing each edge weight by its reciprocal and $\delta$
in $\mathbb{R}^n$ is given by
$\delta^T=(\delta_1,\delta_2,\ldots,\delta_n)$ with
$\delta_i=2-d(i)$, $d(i)$ denotes the degree of the vertex $i$. It
follows easily that the matrix $B$ defined as in Theorem 3.4 is
given by $\displaystyle B=\frac{1}{2}L$. Routine calculations show,
that
    \begin{displaymath}
    (Bx|x)\leq2\sum_{e\in E}\frac{1}{w(e)}, \quad
     \textnormal{for all} \
    x \ \textnormal{in} \ \{-1,1\}^n.
    \end{displaymath}
Moreover we get $\displaystyle
(B\overline{x}|\overline{x})=2\sum_{e\in E}\frac{1}{w(e)}$, for
$\overline{x}^T=(x_1,x_2,\ldots,x_n)$ in $\{-1,1\}^n$, a 2-colouring
of the vertices $1,2,\ldots,n$. So again by Theorem 3.4, part 2 we
get $\displaystyle \Gamma=\left(\sum_{e\in E}
\frac{1}{w(e)}\right)^{-1}$.
\end{rem}


\begin{thebibliography}{1000}

\bibitem{Ref[1]}
  \textsc{R. Bapat, S.J. Kirkland, M.Neumann}, On distance matrices and Laplacians,
  \textit{Linear Algebra and its Applications} \textbf{401} (2005), 193-209.

\bibitem{Ref2}
  \textsc{I. Doust, A. Weston}, Enhanced negative type for finite metric trees,
  \textit{Journal of Functional Analysis} \textbf{254} (2008), 2336-2364.

\bibitem{Ref3}
  \textsc{I. Doust, A. Weston}, Corrigendum to "Enhanced negative type for finite metric trees",
  \textit{Journal of Functional Analysis} \textbf{255} (2008), 532-533.

\bibitem{Ref4}
  \textsc{P. Hjorth, P. Lisonek, Steen Markvorsen and Carsten Thomassen},
  Finite Metric Spaces of Strictly Negative Type$^*$,
  \textit{Linear Algebra and its Applications} \textbf{270} (1998), 255-273.

\bibitem{Ref5}
  \textsc{Peter Nickolas and Reinhard Wolf}, Distance geometry in quasihypermetric spaces. I,
  \textit{Bull. Aust. Math. Soc.} \textbf{80} (2009), 1-25.

\bibitem{Ref6}
  \textsc{Peter Nickolas and Reinhard Wolf}, A., Distance geometry in quasihypermetric spaces. II,
  \textit{to appear in Mathematische Nachrichten}.

\bibitem{Ref7}
  \textsc{Peter Nickolas and Reinhard Wolf}, Distance geometry in quasihypermetric spaces. III,
  \textit{to appear in Mathematische Nachrichten}.

\bibitem{Ref8}
  \textsc{Peter Nickolas and Reinhard Wolf}, Finite quasihypermetric spaces,
  \textit{Acta Math. Hungar.} \textbf{124 (3)} (2009), 243-262.

\bibitem{Ref9}
  \textsc{Anthony Weston}, Optimal lower bound on the supremal strict $p$-negative type of a finite metric space,
  \textit{Bull. Aust. Math. Soc.} \textbf{80} (2009), 486-497.

\end{thebibliography}
\end{document}